\documentclass[12pt]{article}   

\usepackage[margin=25mm]{geometry}

\usepackage{amsfonts}
\usepackage{amssymb}
\usepackage{amsmath}
\usepackage{amsthm}
\usepackage{colortbl}
\usepackage{array}

\usepackage{graphicx}
\usepackage{caption}
\usepackage{hyperref}

\usepackage{natbib} 
\newcommand{\stats}[1]{\mathbb{#1}}
\newcommand{\stat}[3]{\mathbb{#1}_{#2}\!\left[#3\right]}
\newcommand{\matsq}[1]{\mathcal{M}_{#1}(\RR)}
\newcommand{\mat}[2]{\mathcal{M}_{#1,#2}(\RR)}
\newcommand{\vect}[1]{\RR^{#1}}

\newcommand{\Cb}{\mathbf{C}}
\newcommand{\Db}{\mathbf{D}}
\newcommand{\cb}{\mathbf{c}}

\newcommand{\const}[1]{\mathbf{a}_{#1}}
\newcommand{\constO}[1]{\mathbf{t}_{#1}}
\newcommand{\Cc}[2]{\Cb_{#1,#2}}
\newcommand{\D}[2]{\Db_{#1,#2}}

\newcommand{\Cyx}{\Db}
\newcommand{\Cyy}{\Cb}
\newcommand{\nxi}[1]{{d_{#1}}}
\newcommand{\nyi}[1]{{p_{#1}}}
\newcommand{\nui}[1]{{n_{#1}}}
\newcommand{\bbi}[1]{{\bb_{i}}}
\newcommand{\nx}{{d}}
\newcommand{\ny}{{p}}
\newcommand{\ngg}{{m}}
\newcommand{\disc}{h}
\newcommand{\feasibility}{\alpha_{t}}
\newcommand{\xx}{\mathbf{x}}
\newcommand{\yy}{\mathbf{y}}
\newcommand{\YY}{\mathbf{Y}}
\newcommand{\sss}{\mathbf{s}}

\newcommand{\UU}{\mathbf{U}}
\newcommand{\mmu}{\boldsymbol{\mu}}
\newcommand{\SSigma}{\boldsymbol{\Sigma}}
\newcommand{\bzero}{\mathbf{0}}
\newcommand{\bone}{\mathbf{1}}
\newcommand{\bb}{\mathbf{b}}
\newcommand{\QQ}{\mathbf{Q}}
\newcommand{\qq}{\mathbf{q}}
\newcommand{\bid}{\mathbf{I}}
\newcommand{\AAA}{\mathbf{A}}
\newcommand{\dd}{\mathbf{d}}
\newcommand{\balpha}{\boldsymbol{\alpha}}
\newcommand{\bbeta}{\boldsymbol{\beta}}

\newcommand{\PP}{\mathbf{P}}
\newcommand{\RR}{\mathbb{R}}

\newcommand{\iset}{\mathcal{I}}

\newcommand{\diag}{\mathrm{diag}}
\newcommand{\esp}{\mathbb{E}}
\newcommand{\cov}{\mathbb{C}\text{ov}}

\newtheorem{theorem}{Theorem}%  meant for continuous numbers
\newtheorem{proposition}[theorem]{Proposition}% 
\newtheorem{corollary}{Corollary}
\newtheorem{remark}{Remark}%

\usepackage{authblk}   

\title{A scalable problem to benchmark robust multidisciplinary design optimization techniques}

\author[1]{A. Aziz-Alaoui}
\author[2]{O. Roustant}
\author[1]{M. De Lozzo}

\affil[1]{{\small Institut de Recherche Technologique Saint Exup\'ery, 31400 Toulouse, France}}

\affil[2]{{\small Institut de Math\'ematiques de Toulouse, Universit\'e Paul Sabatier, 31062 Toulouse Cedex 9, France}}

\date{} 

\begin{document}
	
\maketitle

\begin{abstract}
A scalable problem to benchmark robust multidisciplinary design optimization algorithms (RMDO) is proposed. This allows the user to choose the number of disciplines, the dimensions of the coupling and design variables and the extent of the feasible domain. After a description of the mathematical background, a deterministic version of the scalable problem is defined and the conditions on the existence and uniqueness of the solution are given. Then, this deterministic scalable problem is made uncertain by adding random variables to the coupling equations. Under classical assumptions, the existence and uniqueness of the solution of this RMDO problem is guaranteed. This solution can be easily computed with a quadratic programming algorithm and serves as a reference to assess the performances of RMDO algorithms. This scalable problem has been implemented in the open source software GEMSEO and tested with two techniques of statistics estimation: Monte-Carlo sampling and Taylor polynomials.
\end{abstract}

%\keywords{Multidisciplinary design optimization, Robust optimization, Scalable problem, Benchmarking, Uncertainty quantification, Quadratic programming}



\maketitle

\section{Introduction}\label{intro}
\textit{Multidisciplinary design optimization} (MDO) aims at designing complex systems composed of several coupled subsystems called \textit{disciplines}. 
The resolution of a MDO problem depends on both an optimization algorithm and a mathematical \textit{formulation} of the optimization problem, also called \textit{architecture} \citep{martinsLambe}. One of the main characteristics of a formulation is how it ensures the coupling between the disciplines. The performance of these techniques can be assessed with popular problems whose dimensions (e.g. the sizes of the variables or the number of disciplines) are either fixed \citep{sobieski-usecase:98, sellar96}, or chosen by the user \citep{gallard17:use-case, tedfordmartins}; in this second case, the problem is said to be \textit{scalable}.

\textit{Uncertainty-based MDO} (UMDO), also called \textit{multidisciplinary robust design optimization} (MRDO) or \textit{robust MDO} (RMDO), is an active and recent MDO research topic \citep{bookbrevault20,yao2011} for which there are only few references problems \citep{liu20umdo}. Thus, in this paper, we propose a scalable problem to benchmark UMDO algorithms, revisiting and extending the deterministic one proposed by \citep{tedfordmartins}. We first give the existence and uniqueness conditions of the solution for the deterministic scalable problem. Then, we extend this problem to the UMDO framework by adding random variables in the coupling equations. Under classical assumptions, we obtain the existence and uniqueness of the solution for the scalable UMDO problem, which can be computed by quadratic programming (QP).

The paper is organized as follows. Section \ref{sec2} describes the mathematical formalism of MDO and UMDO. The scalable problem is presented and studied mathematically in Section \ref{sec3}. In Section \ref{sec4}, we illustrate how this problem can be used in practice by comparing the performances of two techniques for statistics estimation: Monte-Carlo sampling and Taylor polynomials. We give concluding remarks in Section \ref{conclusion}.

\section{MDO background}\label{sec2}

\subsection{MDO problem}\label{sec2.1}

A general \textit{optimization problem} consists in minimizing a \textit{cost function} $f: \mathcal{X}\subset\RR^\nx \to \mathcal{F} \subset \RR$ while satisfying an \textit{inequality constraint} associated with a function $g:\mathcal{X} \to \mathcal{G} \subset \RR^\ngg$:
\begin{equation}\label{eq:opbm}
\begin{aligned}
\min_{\xx}\quad & f(\xx)&\\
\textrm{s.t.} \quad & g\left(\xx\right)\preccurlyeq \bzero
\end{aligned}
\end{equation}
where $\preccurlyeq$ is the component-wise inequality operator.
The optimization variable $\xx$ is often called \textit{design variable} or \textit{control variable} and $f$ is a particular \textit{objective function}.

When the objective and constraint values result from $N$ interdependent sets of equations, the optimization problem (\ref{eq:opbm}) can be replaced by the general \textit{MDO problem} \citep{Balesdent2012ASO}
\begin{equation}\label{General_MDO_pb}
\begin{aligned}
\min_{\xx,\yy,\sss}\quad & f(\xx,\yy,\sss)&\\
\textrm{s.t.} \quad & g_0\left(\xx,\yy,\sss\right)\preccurlyeq \bzero\\
& g_i\left(\xx_0,\xx_i,\yy_i,\sss_i\right)\preccurlyeq \bzero,\quad\forall i \in \iset \\
& \yy_i=\disc_i\left(\xx_0,\xx_i,\yy_{-i}\right)\\
& r_i\left(\xx_0,\xx_i,\yy_{-i},\sss_i\right)= \bzero
\end{aligned}
\end{equation}
with $\iset=\{1,\ldots,N\}$, 
$\xx=(\xx_0^\top,\xx_1^\top,\ldots,\xx_N^\top)^\top$, 
$\yy=(\yy_1^\top,\ldots,\yy_N^\top)^\top$ and 
$\sss=(\sss_1^\top,\ldots,\sss_N^\top)^\top$. The \textit{discipline} 
\begin{equation*}
    \begin{aligned}
    \disc_i:\mathcal{X}_0\times\mathcal{X}_i\times\mathcal{Y}_{-i}&\rightarrow\mathcal{Y}_i\\
    \xx_0,\xx_i,\yy_{-i}&\mapsto\disc_i(\xx_0,\xx_i,\yy_{-i})
    \end{aligned}
\end{equation*}
representing the $i^{\textrm{th}}$ set of equations depends on the design variables $\xx_0\in\mathcal{X}_0\subset\RR^\nxi{0}$ common to all the disciplines and the local design variables $\xx_i\in\mathcal{X}_i\subset\RR^\nxi{i}$ specific to $\disc_i$. Moreover, its output variable $\yy_i\in\mathcal{Y}_i\subset\RR^{\ny_i}$ is constrained to be an input of the other disciplines and it is then called a \textit{coupling variable}. $\disc_i$ depends in turn on all the coupling variables but $\yy_i$: 
$$\yy_{-i}=(\yy_j)_{j\in\iset \backslash \{i\}} \in \mathcal{Y}_{-i}.$$
The design and coupling variables are independent degrees of freedom of the MDO problem (\ref{General_MDO_pb}). Figure \ref{fig:mdf} illustrates the input-output definition of $\disc_i$.\\

\begin{figure}[h]
\centering
\includegraphics[width=0.7\linewidth]{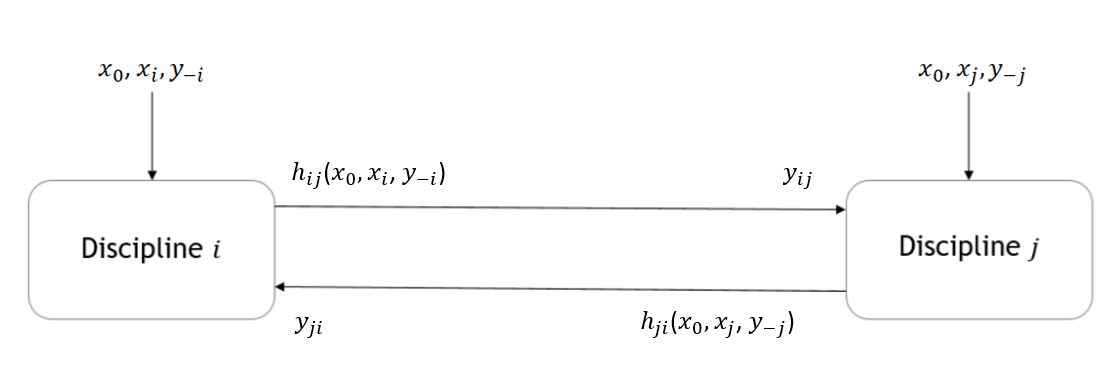}\caption{\label{fig:mdf} \textit{Input-output relationship for two coupled disciplines in a MDO problem.}}
\end{figure}

Lastly, $\disc_i$ depends on specific \textit{state variables} $\sss_i$ through the \textit{state equation} $r_i(\xx_0, \xx_i, \yy_{-i}, \sss_i) = \bzero$ representing the equations of the discipline in their \textit{residual form} \citep{mdobook}.

\begin{remark}
From a numerical point of view, the MDO problem  (\ref{General_MDO_pb}) implies that the coupling equations $\yy_i = \disc_i(\xx_0, \xx_i, \yy_{-i}),~i\in\iset$, must be verified at the end of the optimization process.
\end{remark}

\begin{remark}
A discipline is named so because it rep-
resents either a specific model involved in the
optimization problem, such as a structural analysis or Navier-Stokes equation, or a version of $f$ or $g$ fixing some design variables to handle the remaining ones with a suitable optimizer. These disciplines can vary greatly in complexity and scale.
\end{remark}

\subsection{MDF formulation}\label{sec2.2}

Solving the MDO problem (\ref{General_MDO_pb}) is almost impossible analytically and is often expensive numerically for real-life applications. Reformulating the problem by taking advantage of disciplinary knowledge (gradient, convexity, sub-optimizer, coupling dimension, etc.) is a common practice in MDO. 
Several formulations of the MDO problem (\ref{General_MDO_pb}) have been proposed to make its numerical resolution as efficient as possible \citep{yiPark, martinsLambe}. 
In this article, we will focus on standard \textit{multidisciplinary feasible} (MDF) formulation \citep{cramer:mdf, sobieski:mdf}.

One of the challenges of MDO is to solve the system of coupling equations
\begin{equation} \label{eq:coupling}
    \left\{\yy_{i} = \disc_{i}(\xx_0, \xx_i, \yy_{-i}), \quad i\in\iset\right\}
\end{equation}
This process is called \textit{multidisciplinary analysis} (MDA) in the MDO community \citep{coelhoAl} and is made possible by the implicit function theorem, here expressed in its scalar form for the sake of readability:

\begin{theorem}[\citep{allendoerfer:ift}, \citep{implicitsobieski90}]\label{implicit_function_theorem}
Let $F$ a mapping from $\mathcal{X} \times \mathcal{Y}$ to $\mathcal{Y}$ such that $ F : x,y\mapsto F(x, y) = \disc (x,y) - y$ and $\disc\in\mathcal{C}^1$. Let $(x,y) \in \mathcal{X} \times \mathcal{Y}$ such that $F(x, y) = 0$, and $\frac{\partial F}{\partial y}(x,y) = 0$. 

Then, there exists a $\mathcal{C}^1$-function $c$ defined on an open neighborhood of $(x, y)$ such that $y = c(x)$.
\end{theorem}

The main characteristic of the MDF is that the coupling equations are assumed to be verified. In practice, this results in performing a MDA at each iteration of the optimization algorithm addressing the MDO problem. The simplicity can be balanced by a high computational cost, in particular when the gradients are missing. Furthermore, in the MDF formulation, the state equations are removed, because it is assumed that they have been already solved in a specific optimization problem depending on a single discipline. Therefore, the MDF formulation allows to rewrite the MDO problem (\ref{General_MDO_pb}) as: 
\begin{equation}\label{eq:MDF}
\begin{aligned}
\min_{\xx}\quad & f(\xx_0, c(\xx))&\\
\textrm{s.t.} \quad & g_0\left(\xx_0,c(\xx)\right)\preccurlyeq \bzero\\
& g_i\left(\xx_0,\xx_i,c_i(\xx)\right)\preccurlyeq \bzero,~\forall i \in \iset
\end{aligned}
\end{equation}
where $c_i$ is the $i^{\textrm{th}}$ component of the function $c$ introduced in Theorem \ref{implicit_function_theorem}.

\begin{remark}
For the sake of simplicity, we keep the usual notations $f$, $g_0$ and $g_i$ considered in the original problem (\ref{General_MDO_pb}) for the cost and constraint functions even if they are different mathematical objects, no longer depending on the state variables $s$.
\end{remark}
In practice, solving the system of equations (\ref{eq:coupling}) at a given $\xx$ is done with an iterative scheme. Fixed-point iteration techniques and sub-optimization processes minimizing $\|\yy - \disc(\xx,\yy)\|^{2}$ over $\mathcal{Y}$ are classical kinds of MDA methods \citep{tedford}. Notice that fixed-point methods do not always converge but it is sufficient that $\disc$ defines a contraction mapping according to the Banach's attractive fixed-point theorem \citep{ortega73:fpiconv}.

\subsection{Robust MDO problem}\label{sec2.3}

We consider a MDF-formulated MDO problem where the disciplines depend on a random vector $\UU$ defined over a probability space $(\Omega, \mathcal{A}, \mathbb{P})$. We denote by 
$\mathcal{U}=\UU(\Omega)\subset\RR^\nui{}$ its image set. 
We assume that $\UU$ is square integrable and we denote $\mmu = \esp[\UU]$ its expectation and $\SSigma = \cov[\UU]$ its covariance matrix.
Similarly, for a function $\psi: \mathcal{X}\times\mathcal{U} \to \RR^m$,
we denote  $\mmu_{\psi}=\esp[\psi(\xx,\UU)]$ and $\SSigma_{\psi}=\cov[\psi(\xx,\UU)]$. 
For a given matrix $\mathbf{M}$, 
$\sqrt{\mathbf{M}}$ denotes the matrix obtained from $\mathbf{M}$ by computing the square root element-wise (Hadamard root), 
and $\text{diag}(\mathbf{M})$ is the vector of diagonal terms of $\mathbf{M}$.
Finally, $\sigma$ denotes the element-wise standard deviation: $\sigma(\UU) = \sqrt{\text{diag}(\cov[\UU])}$.

\begin{remark}
In this work, we do not make any other assumption about the probability distribution of $\UU$.
\end{remark}

\subsubsection{Robust optimization problem}

A general \textit{robust optimization problem}  consists in minimizing a cost function $\stat{F}{}{f(\cdot,\UU)}: \mathcal{X} \to \mathcal{F}$ subject to an inequality constraint on a function $\stat{G}{}{g(\cdot,\UU)}$:
\begin{equation}\label{eq:ropbm}
\begin{aligned}
\min_{\xx}\quad & \stat{F}{}{f(\xx,\UU)}&\\
\textrm{s.t.} \quad & \stat{G}{}{g(\xx,\UU)}\preccurlyeq 0
\end{aligned}
\end{equation}
where $\stats{F}$ and $\stats{G}$ are statistics to be chosen according to the uncertainty quantification study.
Recall that $\preccurlyeq$ is a component-wise operator. Thus, 
$\stat{G}{}{g(\xx,\UU)}\preccurlyeq 0$ means that for all components $i=1, \dots, m$ we have $\stat{G}{}{(g(\xx,\UU))_i} \leq 0$.

\subsubsection{Uncertainty quantification}

The practitioners often consider the expectation $\stats{E}$ for $\stats{F}$, which guarantees the robustness in central tendency. Concerning the statistics for the constraints, their choice is often guided by the will to ensure the feasibility of the optimum with a high confidence level. A conservative statistics is the supremum. 
However, its estimation is often prohibitively expensive, because it implies the resolution of a \text{minimax} problem. This worst-case statistics can be replaced by the \textit{vectorial} quantile of order $1 - \alpha$ for a small value of $\alpha \in (0, 1)$  $$\qq_{1-\alpha}(\xx) = \mathbb{Q}\!\left[g(\xx, \UU); 1-\alpha\right]$$ 
defined componentwise with the usual quantile:  
$$\mathbb{P}\!\left[\left[g(\xx, \UU)\right]_i \leq q_{1-\alpha,i}(\xx) \right]= 1 - \alpha \qquad (i=1, \dots, n)$$
Similary, some users are interested by the \textit{vectorial} probabilistic constraint \citep{agarwal}
$$\mathbb{P}\!\left[g(\xx, \UU)\succcurlyeq \bzero\right] \preccurlyeq \bone-\varepsilon$$
where $\varepsilon > 0$.
To save computational time, the estimation of probabilities and quantiles may be replaced by a combination of the expectation and the variance, sometimes called \textit{margin} \citep{giassi}:
$$\mmu_{g} + \kappa \sqrt{\text{diag}(\SSigma_{g})}$$
where $\kappa \in \RR$.
This statistics can be viewed as an approximation of a quantile $\stats{Q}\left[g(\xx,\UU); 1 - \alpha_{\kappa}\right]$  where $\kappa$ is the $1-\alpha_{\kappa}$ quantile of the standard normal distribution. This approximation may be relevant when $g(\xx, \UU)$ is close to a multivariate normal distribution, which happens for instance when $g$ is linear with respect to $\UU$ and when $\UU$ is normally distributed.

\subsubsection{MDF-based robust MDO}

The uncertainty-based version of the MDO problem (\ref{General_MDO_pb}) considers a random variable $\UU_0$ common to all the $N$ disciplines and a random variable $\UU_i$ specific to the $i^{\textrm{th}}$ discipline. We denote $\UU=(\UU_0,\UU_1,\ldots,\UU_N)$ the whole input random vector. The general UMDO problem \citep{yao2011} can be written
\begin{equation}\label{General_UMDO_pb}
\begin{aligned}
\min_{\xx,\YY,\sss}\quad & \stat{F_{\UU}}{}{f(\xx_0,\UU_0,\YY,\sss)}&\\
\textrm{s.t.} \quad &
\stat{G}{0}{g_0(\xx_0,\UU_0,\YY,\sss)}\preccurlyeq \bzero\\
& \stat{G}{i}{g_i(\xx_0,\xx_i,\UU_0,\UU_i,\YY_i,\sss_i)}\preccurlyeq \bzero,\quad\forall i \in \iset \\
& \YY_i=\disc_i\left(\xx_0,\xx_i,\UU_0,\UU_i,\YY_{-i}\right), \quad \forall \omega \in \Omega\\
& r_i\left(\xx_0,\xx_i,\UU_0,\UU_i,\YY_i,\sss_i\right)= \bzero, \quad \forall \omega \in \Omega
\end{aligned}
\end{equation}

\noindent where $\YY$ denotes the random coupling vector.
\begin{remark}
Solving the UMDO problem (\ref{General_UMDO_pb}) is quite hard because the probability distribution of $\YY$ is unknown. In practice, the problem is simplified with $\YY$-based deterministic variables, e.g. realizations, statistical moments or distribution hyperparameters, combined with consistency constraints.
\end{remark}

A robust version of the MDF formulation (\ref{eq:MDF}) can be written \citep{Koch:2002} as
\begin{equation}\label{eq:UMDF}
\begin{aligned}
\min_{\xx}\quad &\stat{F}{}{f\left(\xx_0,\UU_0,c(\xx,\UU)\right)}\\
\textrm{s.t.} \quad & \stat{G}{_0}{g_0\left(\xx_0,\UU_0,c(\xx,\UU)\right)}\preccurlyeq\bzero\\
& \stat{G}{i}{g_i\left(\xx_0,\xx_i,\UU_0,\UU_i,c_i(\xx,\UU)\right)}\preccurlyeq \bzero,\forall i \in \iset
\end{aligned}
\end{equation}
where $c:\mathcal{X}\times\mathcal{U}\mapsto \mathcal{Y}$ is a $\mathcal{C}^1$-function defined in an open neighborhood of $(\xx,\UU(\omega)),~\omega\in\Omega,$ such that $c(\xx,\UU(\omega))=\YY(\omega)$ for all $\omega\in\Omega$. 

As in the deterministic case, the MDF formulation is popular to solve MDO problems, due to its ease of implementation. However, the statistics appearing in the problem (\ref{eq:UMDF}) rarely have analytical expressions and need to be estimated. In addition, the coupling equations (\ref{eq:coupling}) must be satisfied over the whole probability space:
$$\forall \omega \in \Omega,\quad\YY(\omega)=\disc\left(\xx,\YY(\omega)\right).$$

\subsubsection{Statistics estimation with Monte Carlo sampling}\label{sec:mcstats}

Monte Carlo (MC) sampling is a classical technique to propagate the uncertainties through the disciplines while performing an MDA for each realization of $\UU$ \citep{OAKLEY199815}. 
Let us consider $\UU^{(1)},\ldots,\UU^{(M)}$ a $M$-sample of $\UU$; these random variables are independent and identically distributed as $\UU$. The unbiased estimators of the expectation and variance  of some function $\psi:\mathcal{X}\times\mathcal{U}\mapsto\RR$ read
$$\hat{\mu}_{\psi}(\xx)=\frac{1}{M}\sum_{i=1}^M \psi\left(\xx,\UU^{(i)}\right)$$
$$\hat{\sigma}^2_{\psi}(\xx)=\frac{1}{M-1}\sum_{i=1}^M \left(\psi(\xx,\UU^{(i)})-\hat{\mu}_{\psi}(\xx)\right)^2.$$ 
$\hat{\mu}_{\psi}(\xx)$
and $\hat{\sigma}^2_{\psi}(\xx)$ converge slowly in distribution to Gaussian distributions with $\mathcal{O}\left(M^{-1}\right)$ rate, involving a hundred times more samples to improve the estimation accuracy by a factor of ten. Costly in general, MC sampling may become prohibitive  with the MDF formulation whose numerical implementation implies three nested loops: a MDA loop in a sampling loop itself in an optimization loop. By denoting $\gamma_k L$ the MDA loop length at the $k^{\textrm{th}}$ iteration of the optimizer and $\rho_k M$ the sample size, the number of evaluations of $\psi$ is equal to $\sum_{k=1}^K \gamma_k\rho_kLM$ where $K$ is the size of the optimization loop and $(\gamma_k,\rho_k)\in]0,1]^2$. Thus, the required number is bounded by $KLM$ which may not be so pessimistic in some cases \citep{mahadevan:2000}.

MC sampling may also be used to estimate the probability $\mathbb{P}\!\left[\psi(\xx, \UU)\succcurlyeq 0\right]$ \citep{Sobol1994APF}:
$$\hat{\mathbb{P}}_\psi = \frac{1}{M} \sum_{i = 1}^{M} \mathbf{1}_{\psi(\xx, \UU)\succcurlyeq 0}.$$ 
It demands a tremendous budget to assess a small probability. For instance, a $10^{k+2}$-sample is needed to guarantee a $10^{-k}$ estimation of $\mathbb{P}$ with $10\%$ of variation \citep{silverman_cmc}.

Despite these limitations, we choose MC sampling because in addition to its simple implementation in the Python library GEMSEO on which we worked, its precision can be theoretically controlled by increasing the sample size, which is useful for building UMDO benchmarking problems and for comparing it to advanced statistics estimation methods.

\subsubsection{Statistics estimation with Taylor polynomials}

To drastically reduce the computational cost, we propose to perform a unique MDA, approach it with a Taylor polynomial (TP) and deduce analytical statistics. 

The first-order TP of some function $\psi(\xx,\cdot):\mathcal{U}\mapsto\RR$ around $\mmu$ is
$$\hat{\psi}(\xx, \UU) = \psi(\xx, \mmu) + (\UU - \mmu)^\top\nabla_U\psi(\xx, \UU).$$

Then, the estimators of the first and second central moments of $\hat{\psi}(\xx,\UU)$ read
$$\hat{\mmu}_{\psi}(\xx)=\psi(\xx,\mmu)$$
$$\hat{\SSigma}_{\psi}(\xx,\UU)=\nabla_U\psi(\xx, \UU)^\top\SSigma\nabla_U\psi(\xx, \UU)$$
and the vector of variances of the components of $\psi(\xx,\UU)$:
$$\widehat{\sigma_{\psi}^2}(\xx,\UU)=\diag\left(\hat{\SSigma}_{\psi}(\xx,\UU)\right).$$

The counterpart of the cheapness of this method is that the polynomial approximation is only correct locally, making the accuracy of the estimators collapse as the variances of the components of ${\UU}$ increase \citep{Arras1998AnIT}. Moreover, it requires the evaluation of partial derivatives which can be costly if the analytical gradients are missing. Higher order TPs would improve the quality of the approximation of $\psi(\xx,\cdot)$ but would need the evaluation of the Hessian which is rarely available. Moreover, the final goal is not to estimate $\psi(\xx,\cdot)$ but the first and second moments of $\psi(\xx,U)$ for which low order TPs can be sufficient, except in case of very strong non-linearity.

\section{A scalable problem to benchmark UMDO algorithms} \label{sec3}
In what follows, we start with a scalable MDO problem found in the literature. Then, we propose a judicious rewriting to transform it into a classical quadratic optimization problem under linear constraints whose solution can be determined analytically. From there, we modify the scalable problem to take into account uncertain parameters. We show that under certain conditions on the expressions of the constraints, it is always possible to obtain a deterministic analytical solution. 

\subsection{A scalable MDO problem} \label{sec3.1}

\citep{tedfordmartins} proposed a scalable MDO problem over the unit design space $\mathcal{X}=[0,1]^\nx$ whose number of disciplines and variable dimensions are chosen by the user:

\begin{equation}\label{eq:martinspb}
\begin{aligned}
\min_{\xx\in\mathcal{X}}\quad &\xx_0^\top\xx_0 + \sum_{i\in\iset}\yy_i^\top\yy_i\\
\textrm{s.t.} \quad & \constO{i} -\yy_i \preccurlyeq 0, \quad \forall i \in \iset\\
\textrm{where} \quad & \yy_i = \const{i} - \D{i}{0}\xx_0 - \D{i}{i}\xx_i + \sum_{j\in\iset \backslash \{i\}}\Cc{i}{j}\yy_{j}
\end{aligned}
\end{equation}

\noindent with 
$\constO{i}, \const{i} \in\vect{\nyi{i}}$.
The coupling variables $\yy_i$ depend linearly on the shared design variables $\xx_0$ with coefficients $\D{i}{0}\in\mat{\nyi{i}}{\nxi{0}}$, on the local design variables $\xx_i$ with coefficients $\D{i}{i}\in\mat{\nyi{i}}{\nxi{i}}$ and on the other coupling variables with coefficients $\Cc{i}{j}\in\mat{\nyi{i}}{\nyi{j}},~j\in\iset \backslash \{i\}$.

\begin{remark}
This problem is said to be \textit{scalable} because the user can set the number of disciplines and the dimensions of the design and coupling variables. This allows to compare the efficiency of coupling algorithms or optimizers for different problem dimensions.
\end{remark}

\subsection{Rewriting as a quadratic programming problem}\label{sec3.2}

The problem (\ref{eq:martinspb}) can be rewritten in a compact form as

\begin{equation*}
\begin{aligned}\label{eq:martinscpl}
\min_{\xx\in\mathcal{X}}\quad &\xx^\top\QQ_{\xx_{0}}\xx + \yy^\top\yy\\
\textrm{s.t.} \quad & \constO{}-\yy \preccurlyeq 0, \quad \forall i \in \iset\\
\textrm{with} \quad & \Cyy \yy = \const{} - \Cyx \xx 
\end{aligned} 
\end{equation*}
\\
\noindent where $\const{}$ and $\constO{}$ are the block vectors in $\vect{\ny}$ obtained by stacking:%the block matrices
\begin{eqnarray*}
\const{} &=& (\const{1}^\top, \dots, \const{N}^\top)^\top,\\
\constO{} &=& (\constO{1}^\top, \dots, \constO{N}^\top)^\top
\end{eqnarray*}
and $\QQ_{x_{0}}, \Cyx$ and $\Cyy$ are the block matrices defined by
$$
\QQ_{x_{0}} = \begin{pmatrix}
\bid_{\nxi{0}} & \bzero \\
\bzero  & \bzero \\
\end{pmatrix}\in\matsq{\nx},
$$
$$
\Cyx = 
\begin{pmatrix}
\D{1}{0} & \D{1}{1} & \cdots & \bzero \\
\vdots  & \vdots  & \ddots & \vdots  \\
\D{N}{0} & \bzero & \cdots & \D{N}{N} 
\end{pmatrix} \in \mat{\ny}{\nx}
$$
and
$$\Cyy = 
\begin{pmatrix}
\bid & -\Cc{1}{2} & \cdots & -\Cc{1}{N} \\
-\Cc{2}{1} & \bid & \ddots & -\Cc{2}{N} \\
\vdots & \ddots & \ddots & \vdots \\
-\Cc{N}{1} & -\Cc{N}{2} & \ldots & \bid \\
\end{pmatrix} \in \matsq{\ny}.
$$

The existence of a solution to the MDO problem (\ref{eq:martinspb}) requires the invertibility of the coupling matrix $\Cyy$, which corresponds to solving the MDA problem. This assumption leads to $\yy = \balpha + \bbeta \xx$ where $\bbeta = -\Cyy^{-1}\Cyx$ and $\balpha = \Cyy^{-1}\const{}$. 
The optimization problem can then be expressed explicitly as a  quadratic programming problem with a quadratic cost function and linear constraints:

\begin{equation}\label{QP_problem}
\begin{aligned}
\min_{\xx\in\mathcal{X}}\quad &\frac{1}{2}\xx^\top\QQ\xx + \cb^\top\xx + \dd\\
\textrm{s.t.} \quad & \AAA\xx \preccurlyeq \bb
\end{aligned}
\end{equation}

\noindent with
$\QQ = 2\left(\QQ_{x_{0}} + \bbeta^\top\bbeta\right)$,
$\cb = 2 \bbeta^\top\balpha$,
$\dd = \balpha^\top\balpha$, 
$\AAA = -\bbeta  
$
and
$\bb = \balpha-\constO{}$.

As the matrix $Q$ is symmetric positive semi-definite\footnote{$\forall \xx \in \RR^{\nx}\setminus\{0_{\nx}\}, \xx^\top\QQ\xx = \xx^\top\QQ_{\xx_0}\xx + \xx^\top\bbeta^\top\bbeta \xx = \|\xx_0\|^2 + \|\bbeta \xx\|^2 \geq 0$.}, the optimization problem (\ref{QP_problem}) is convex and admits a global minimum when the feasible set $\{\xx\in\vect{\nx}: \AAA\xx \preccurlyeq \bb\}$ is not empty. A condition for this minimum to be unique is when $Q$ is symmetric positive definite. This can be achieved when for all the disciplines, the dimension of the coupling variable $\yy_i$ is greater or equal to the dimension of the design variable $\xx_i$. This result is true for any values of the coefficients of $\const,~ \constO{},~ \Cyx,~ \Cyy$.

\begin{proposition} \label{prop:Qsymdefpos}
Let $\const,~ \constO{},~ \Cyx,~ \Cyy$ be uniform random matrices or vectors, i.e. whose elements are independent realizations of a standard uniform variable on [0, 1]. 
If $\forall i\in\iset, ~\nyi{i}\geq \nxi{i}$ and $\ny \geq \nx$, then $Q$ is positive definite with probability 1.
\end{proposition}

\begin{proof}
For simplicity we use the notation (a.s.), standing for almost surely, to state that a property is true with probability $1$.\\
Let us temporarily admit that the rank of $\Cyx$ is $d=d_0 +\sum_{i=1}^N d_i$ (a.s.).
Then, as $\Cyy$ is assumed invertible, the rank of $\bbeta$ is also equal to the rank of $\Cyx$ (a.s.) which is equal to $d$. 
By a standard property, this implies that the Gram matrix $\bbeta^\top \bbeta$ is positive definite (a.s.).
Adding the positive semidefinite matrix $\QQ_{\xx_0}$ preserves positive definiteness. 
Thus $\QQ = 2\left(\QQ_{x_{0}} + \bbeta^\top\bbeta\right)$ is positive definite (a.s.).\\
To show the result about the rank of $\Cyx$, 
we first prove that a $p \times d$ uniform random matrix $\mathbf{M}$ has full rank (a.s.).
Assume for instance that $p \geq d$. The proof is by induction on $d$. 
If $d=1$, as $M_{1, 1}$ is drawn uniformly on $[0, 1]$, $\mathbb{P}(M_{1, 1} \neq 0) = 1$. Thus $\mathbf{M}$ has rank $1$ (a.s.).
Let us assume that the property is valid for any uniform random matrix with $d-1$ columns.
Then, the vector space $\mathcal{V}$ spanned by the last $d-1$ columns $\mathbf{M}_2, \dots, \mathbf{M}_d$ has dimension $d-1$ (a.s.). 
Then by conditioning, in order to prove that $\mathbf{M}$ has rank $d$ (a.s.), it is sufficient to show that 
$\mathbb{P}(\mathbf{M}_1 \notin \mathcal{V}) = 1$, 
when $\mathcal{V}$ is known (i.e. assumed deterministic). This latter property is true because $\mathcal{V}$ has dimension $d-1 < p$, and the law of $\mathbf{M}_1$ is absolutely continuous with respect to the Lebesgue measure in $\mathbb{R}^p$.\\
Using this result, and the assumption that $p_i \geq d_i$, each $\Cyx_{i,i}$ has rank $d_i$ (a.s.). 
Consequently, the block-diagonal submatrix of $\Cyx$ 
$$ \mathbf{N} = \begin{pmatrix}
\D{1}{1} & \cdots & \bzero \\
\vdots  & \ddots & \vdots  \\
\bzero & \cdots & \D{N}{N} 
\end{pmatrix}$$
%formed by the last $\sum_{i=1}^N d_i$ columns 
has rank $\sum_{i=1}^N d_i$ (a.s.).\\
It remains to show that when we join the submatrix $\mathbf{M} := (\Cyx_{1,0}^\top, \dots, \Cyx_{N, 0}^\top)^\top$ formed by the first $d_0$ columns of $\Cyx$, the rank of $\Cyx = [\mathbf{M}, \mathbf{N}]$ increases by $d_0$ (a.s.). 
The proof is by induction on $d_0$ and uses the same arguments as above. For instance, when $d_0 = 1$, it is sufficient to show that 
$\mathbb{P}(\mathbf{M}_1 \notin \mathcal{V}) = 1$
where $\mathcal{V}$ is the space spanned by the columns of $\mathbf{N}$, supposed fixed.
This is true because $\mathrm{dim}(\mathcal{V}) \leq d-1 < p$ and the law of $\mathbf{M}$ is absolutely continuous w.r.t. the Lebesgue measure in $\mathbb{R}^p$. 
\end{proof}

Quadratic programming problems of the general form (\ref{QP_problem}) can be solved algorithmically in polynomial time  using a large-range of techniques, e.g. ellipsoid method, Lagrangian duality or interior points \citep{KOZLOV1980223, QPlagrangian05, QPinteriorpts04}.

\subsection{Tuning the domain of feasibility}\label{sec3.3}
The proposed benchmark is not directly usable in practice. Indeed, 
$\constO{}$ being fixed, it might be possible to have cases where either the constraints cannot be satisfied or either the problem is always feasible which makes the constraints useless. 
An idea would be to set 
%$t$ 
$\constO{}$ 
from bounds of $\yy_i$. However, this is hardly doable because $\yy_i$ depends on the inverse of $\Cyy$.\\
To overcome this issue, we first define $\constO{}$ with a single real parameter $t \in \mathbb{R}$ by $\constO{} = (t, \dots, t) \in \mathbb{R}^p$. Then,
we set $t$ such that the fraction of the design space on which the constraints are satisfied is equal to a given level $\feasibility$:
\begin{equation*}
    \int_{\xx \in \mathcal{X}} \mathbf{1}_{\forall i,j,~ \yy_{ij}(\xx) \geq t} d\xx  = \feasibility
\end{equation*}
Equivalently, this fraction is equal to the probability of satisfying the constraints $\yy_{ij} \geq t$, where $\mathbf{X}$ is uniform on $\mathcal{X}$, and the feasibility condition is written :
\begin{equation*}
    \mathbb{P}_{\mathbf{X}}\!\left[\min_{i,j}~\yy_{ij}(\mathbf{X}) \geq t\right] = \feasibility
\end{equation*} 
Therefore, $t$ is set as the $1 - \feasibility$ quantile of $\min_{i,j}~\yy_{ij}(\mathbf{X})$.

\subsection{Extension to MDF under uncertainty}\label{sec3.4}

In this section, we propose an extension of the parametric MDO problem (\ref{eq:martinspb}) by adding uncertain terms in the expressions of the disciplines:

\begin{equation}\label{TM_UMDO_initial}
\begin{aligned}
\min_{\xx\in\mathcal{X}}\quad &\stat{F}{}{\xx_{0}^\top\xx_{0} + \sum_{i\in\iset}\YY_i^\top\YY_i}\\
\textrm{s.t.} \quad & \stat{G}{}{\constO{i} - \YY_i}\preccurlyeq \bzero, ~\forall i\in\iset\\
\textrm{with} \quad & \YY_i = \const{i} - \D{i}{0}\xx_0 - \D{i}{i}\xx_i\\ &\quad+ \sum_{j\in\iset \backslash \{i\}}\Cc{i}{j}\YY_j + \UU_{i}
\end{aligned}
\end{equation}

\noindent where $\UU_1,\ldots,\UU_N$ are independent random vectors 
%with \mdl{zero} mean \mdlr{$\mmu_1, \ldots, \mmu_N$} and 
with covariance matrices $\Sigma_1,\ldots,\Sigma_N$. 
Without loss of generality, we assume that the $\UU_i$'s are centered (up to a replacement of $\const{i}$ by $\const{i} + \mmu_i$, and $\UU_i$ by $\UU_i - \mmu_i$).
The statistics $\mathbb{F}$ and $\mathbb{G}$ will be defined later. 

The robust MDO problem (\ref{TM_UMDO_initial}) can be rewritten in a more compact way:
\begin{equation}\label{TM_UMDO_compact}
\begin{aligned}
\min_{\xx\in\mathcal{X}}\quad &\stat{F}{}{\xx_{0}^\top\xx_{0} + \sum_{i\in\iset}\YY_i^\top\YY_i}\\
\textrm{s.t.}  \quad& \stat{G}{}{\constO{} - \YY}\preccurlyeq \bzero\\
\textrm{with} \quad & \Cyy \YY = \const{} - \Cyx \xx + \UU
\end{aligned}
\end{equation}

\smallbreak
\noindent where $\UU = (\UU_1 \ldots \UU_N)$ is a random vector with zero mean and block diagonal covariance matrix:

\begin{equation*}
    \SSigma =  \begin{pmatrix}
\SSigma_{1} & \cdots & \bzero \\
\vdots & \ddots & \vdots  \\
\bzero & \cdots & \SSigma_{N}  \\
\end{pmatrix}\in\mat{\ny}{\ny}
\end{equation*}

When $\Cyy$ is invertible, the random coupling vector is written $\YY = \Cyy^{-1}\const{} - \Cyy^{-1}\Cyx + \Cyy^{-1}\UU$. 
Similarly to the computations of Section~\ref{sec3.1} and using the notations therein, we see that the UMDO problem (\ref{TM_UMDO_compact}) becomes a robust optimization problem: 
\begin{equation}\label{TM_UMDO}
\begin{aligned}
\min_{\xx\in\mathcal{X}}\quad &\stat{F}{}{\frac{1}{2}\xx^\top\QQ\xx + \cb^\top\xx + \dd + 
%\UU^\top\PP^\top\PP\UU}\\
\UU^\top (\Cyy^{-1})^\top \Cyy^{-1} \UU}\\
\textrm{s.t.} \quad & \stat{G}{}{\AAA\xx - \bb - \PP \UU}\preccurlyeq 0
\end{aligned}
\end{equation}
%where $\PP = \Cyy^{-1}$.\\

In the sequel, we consider the usual case of the expectation for the objective $\mathbb{F}$ and two cases for the constraints $\mathbb{G}$. For simplicity, we denote $\PP = \Cyy^{-1}$.

The first one is a conservative margin defined from the expectation and the standard deviation and parameterized by a factor $\kappa\in\RR$:
\begin{equation}\label{TM_UMDO_margin}
\begin{aligned}
\min_{\xx\in\mathcal{X}}\quad &\frac{1}{2}\xx^\top\QQ\xx + \cb^\top\xx + \dd + \stat{E}{}{\UU^\top\PP^\top\PP\UU}\\
\textrm{s.t.} \quad & \stat{E}{}{\AAA\xx - \bb - \PP \UU} + 
\kappa \sigma \left[ \AAA\xx - \bb - \PP \UU \right] 
\preccurlyeq 0 
\end{aligned}
\end{equation}

The second one is a probability of violating the constraints, and is parameterized by a level $\varepsilon\in[0,1]$:
\begin{equation}\label{TM_UMDO_proba}
\begin{aligned}
\min_{\xx\in\mathcal{X}}\quad &\frac{1}{2}\xx^\top\QQ\xx + \cb^\top\xx + \dd + \stat{E}{}{\UU^\top\PP^\top\PP\UU}\\
\textrm{s.t.} \quad & \stat{P}{}{\AAA\xx - \bb - \PP  \UU \succcurlyeq \bzero} - \varepsilon \leq 0
\end{aligned}
\end{equation}

Propositions (\ref{thm:margin}) and (\ref{thm:proba}) show that the robust optimization problems (\ref{TM_UMDO_margin}) and (\ref{TM_UMDO_proba}) 
%can be reduced to general
are equivalent to usual quadratic programming problems. 
Therefore, their solutions can be computed efficiently
%easily approached 
with dedicated numerical optimizers. % and serve as oracles.

\begin{proposition}
\label{thm:margin}
The robust optimization problem (\ref{TM_UMDO_margin}) %(\ref{TM_UMDO}) 
reduces to the quadratic optimization problem:
\begin{equation}\label{QP_problem_margin}
\begin{aligned}
\min_{x\in\mathcal{X}}\quad &\frac{1}{2}\xx^\top\QQ\xx + \cb^\top\xx + \dd + \stat{E}{}{\UU^\top\PP^\top\PP\UU}\\
\textrm{s.t.} \quad & \AAA\xx \preccurlyeq \bb-\kappa\sqrt{\diag\left(\PP\SSigma \PP^\top\right)}
\end{aligned}
\end{equation}
\end{proposition}

\begin{proof}
Let $\kappa\in\RR$ and $\UU$ a centered random variable with covariance matrix $\SSigma$. Then,
\begin{flalign*}
&\stat{E}{}{\AAA\xx - \bb - \PP \UU} + 
%\kappa\sqrt{\cov{}[{\AAA\xx - \bb - \PP \UU}]}
\kappa \sigma \left[ \AAA\xx - \bb - \PP \UU \right]
\\
=&\AAA\xx - \bb + 
%\kappa \sqrt{\cov{}[{\PP \UU}]}
\kappa \sigma \left[\PP \UU \right]
\\
=&\AAA\xx - \bb + \kappa \sqrt{\diag(\PP\SSigma \PP^\top)}
\end{flalign*}
\end{proof}

\begin{proposition}
\label{thm:proba}
The robust optimization problem (\ref{TM_UMDO_proba}) reduces to the quadratic optimization problem:
\begin{equation}\label{QP_problem_proba}
\begin{aligned}
\min_{\xx\in\mathcal{X}}\quad &\frac{1}{2}\xx^\top\QQ\xx + \cb^\top\xx + \dd + \stat{E}{}{\UU^\top\PP^\top\PP\UU}\\
\textrm{s.t.} \quad & \AAA\xx \preccurlyeq \bb + \mathbf{q}_{\varepsilon}    \\
\end{aligned}
\end{equation}
where $\mathbf{q_{\varepsilon, i}}$ is the $\varepsilon$-quantile of the distribution of $(\PP \UU)_i$.
\end{proposition}

\begin{proof}
Let $\varepsilon\in[0,1]$. Then, 
\begin{flalign*}
&\quad\mathbb{P}\left[\AAA\xx - \bb - \PP \UU \succcurlyeq 0\right] \leq \varepsilon\\
\Leftrightarrow&\quad\mathbb{P}\left[\PP \UU\preccurlyeq \AAA\xx - \bb \right] \leq \varepsilon\\
\Leftrightarrow &\quad \AAA\xx - \bb \preccurlyeq \qq_{\varepsilon}
\end{flalign*}
where $\mathbf{q_{\varepsilon, i}}$ is the $\varepsilon$-quantile of the distribution of $(\PP\UU)_i$.
\end{proof}

\begin{corollary}
\label{cor:proba_gauss}
When $\UU$ is normally distributed, the robust optimization problem (\ref{TM_UMDO_proba}) reduces to the quadratic optimization problem:
\begin{equation}
\begin{aligned}
\min_{\xx\in\mathcal{X}}\quad &\frac{1}{2}\xx^\top\QQ\xx + \cb^\top\xx + \dd + \stat{E}{}{\UU^\top\PP^\top\PP\UU}\\
\textrm{s.t.} \quad & \AAA\xx \preccurlyeq \bb+\sqrt{\diag(\PP  \SSigma \PP^\top)} \mathbf{q}_{\varepsilon}    \\
\end{aligned}
\end{equation}
where $\mathbf{q_{\varepsilon}}$ is the $\varepsilon$-quantile of the standard normal distribution.
\end{corollary}

\section{Numerical experiments} \label{sec4}

In this section, we show how the problem (\ref{TM_UMDO_initial}) can be used to benchmark statistical estimators in the frame of UMDO, namely MC-based estimators and TP-based estimators.
We only focus on solving the conservative margin problem (\ref{TM_UMDO_margin}) because the probability case (\ref{TM_UMDO_proba}) based on MC sampling is too costly, as discussed in Section \ref{sec:mcstats}. Some efficient methods would be more appropriate to estimate the probabilities, such as FORM/SORM \citep{form_sorm} or importance sampling \citep{ISampling}. However these techniques are not available at the moment in the MDO software GEMSEO (see \ref{sec:software}). Furthermore, an exhaustive comparison of estimators is out of the scope of this work proposing a scalable benchmark problem for MDO under uncertainty.

The solution of the QP problem (\ref{QP_problem_margin}) used as a reference is computed with the interior-point method \citep{cxvopt_theoric}.

\subsection{Problem configuration}\label{sec:problem_configuration}
We consider the scalable problem (\ref{TM_UMDO_initial}) with $N = 2$ disciplines sharing $\nxi{0} = 1$ design variable. Each discipline considers two local design variables and outputs three coupling variables, i.e. $\nxi{1}=\nxi{2}=2$ and $\nyi{1}=\nyi{2}=3$. The dimensions  satisfy the conditions of Proposition \ref{prop:Qsymdefpos}. 

For the sake of simplicity, we consider the uncertain variables as independent and centered Gaussian variables with a standard deviation equal to $0.01$. This magnitude has a realistic order, considering that the design variables belong to the unit hypercube.

In addition, the feasibility level $\feasibility$ is set to $0.5$, which means that  half of the design space satisfies the constraints as explained in Section \ref{sec3.3}.

Lastly, regarding the constraints, we take $\kappa=2$ in the definition of the margin (\ref{TM_UMDO_margin}).

\subsection{Numerical aspects}

\subsubsection{UQ settings}

For the MC-based estimators, the problem is solved with a sample of size $M = 200$. The experience is repeated 20 times in order to assess the estimation error. The results are expressed in terms of mean and standard deviation.

For the TP-based estimators, we consider first-order TP using the analytical gradients. 

\subsubsection{MDO settings}

To solve the MDO problem, we consider the gradient-free optimization algorithm COBYLA (constrained optimization by linear approximation) \citep{Powell2007AVO} with a maximum of $100$ iterations, in combination with the MDF formulation. The latter uses the Jacobi algorithm to perform the MDA, with a tolerance of $10^{-4}$ and a maximum of $30$ iterations in order to ensure the multidisciplinary feasibility at each iteration of the optimization. The relative tolerance of COBYLA algorithm for design variables and objective function is set to $10^{-8}$, and the tolerance applied on the inequality constraints is set to $10^{-4}$.

\subsubsection{Software}\label{sec:software}

We used and contributed to the open source Python library GEMSEO\footnote{https://www.gemseo.org/} \citep{gemseo_paper}.  This software allows to easily define a MDO scenario in terms of design space, disciplines, objective and constraints, to choose a formulation and to solve the related optimization problem. We implemented the robust MDO framework and added new capabilities to instantiate the scalable problem for both deterministic and uncertainty cases. To solve the equivalent quadratic programming problem that serves as a reference, we rely on the quadratic programming library CVXOPT that implements the interior-point method  \citep{cvxopt}.

\subsection{Results}

\begin{table*}

\begin{center}
\begin{tabular}{|l|c|c|c|} \hline 
  &  $\bf{\Delta_\mathbf{x} (\%)}$ & $\bf{\Delta_f} (\%)$ 
 & $\bf{\Delta_g} (\%)$ \\ \hline \hline
\bf{MC} & $\underset{\small{(0.176)}}{0.370}$ & $\underset{\small{(0.127)}}{0.592}$ & $\underset{\small{(0.278)}}{0.877}$  \\ \hline
\bf{TP} & $0.093$  & $0.005$  & $0.143$ \\ \hline
\end{tabular}
\end{center}
\vspace{1em}
\caption{Percentage errors of the numerical solutions to the problem (\ref{TM_UMDO}) configured as stated in Section \ref{sec:problem_configuration} when using a constraint of type \textit{margin}; expressed as $100\times\|\textrm{estimation}-\textrm{reference}\|/\|\textrm{reference}\|$. The solutions have been obtained with $M = 200$ samples (MC) and first-order Taylor polynomials (TP).
For MC, the experience has been repeated 20 times and the table displays the mean  over these repetitions, together with the standard deviation (in brackets).}
\label{tab:margin}
\end{table*}

Table \ref{tab:margin}  compares the MC estimators and the TP ones in terms of percentage estimation errors of the optimal design vector, objective and constraint, expressed as  $$100\times\frac{\|\textrm{estimation}-\textrm{reference}\|}{\|\textrm{reference}\|}.$$

The estimation of these optimal quantities results from the MDO problem resolution with GEMSEO while the reference solution is obtained with CVXOPT applied to its QP counterpart problem (\ref{QP_problem_margin}).

\paragraph{Validation of the implementation}

The results show that the error of the MC-based estimator is lower than $1\%$ with an affordable sample size here $M=200$. Thus, this estimator converges to the reference solution which confirms the theoretical result presented in Proposition \ref{thm:margin} and validates our implementation of the scalable problem. 
Thereby, this scalable problem can be used to benchmark statistic estimation algorithms.

\paragraph{Comparison of the estimators}

As expected, the MC-based estimator method remains costly as MDF-based robust MDO implies as many MDAs as new samples. Thus, when the convergence of the MDA algorithm requires 10 iterations, increasing the sample size $M$ by a factor of 10 increases the number of discipline evaluations by a factor of 100. Yet, warm-start methods implemented in GEMSEO have been used to speed up the convergence. 

On the other hand, TP method performs better than the MC method. Furthermore this technique requires only one resolution of the MDA per iteration of the optimizer. This makes the robust MDO problem resolution as cheap as its deterministic counterpart.

\section{Conclusion and future work} \label{conclusion}
In this paper, we revisited a deterministic scalable problem in the literature and extended it to the frame of uncertainty quantification. We rewrote it as a quadratic problem with linear constraints and gave a sufficient condition for the existence and uniqueness of its solution. This solution can be efficiently computed with QP algorithms. Thus, it can be used as a reference to benchmark MDO algorithms. We showed that, when the constraints are either probabilities or margins, the scalable problem under uncertainty collapses to a similar QP problem with a unique and known solution.
As an illustration, we used this scalable problem to benchmark two techniques based on the MDF formulation: Monte-Carlo sampling, and Taylor polynomials. The second method proved to be relevant to get a first idea of the solution at a very low cost, namely the cost of solving the deterministic scalable problem. 

More generally, the scalable problem can be used to benchmark any kind of algorithms such as MDO formulations.
In particular, this scalable problem could be useful to compare the MDF formulation with multilevel ones, closer to the industrial design process involving sub-optimization problems. The interest would be also to benchmark these methods by varying the dimension of the problem. Concerning the scalable problem itself, we could extend it to non-linear relations, by using non-linear regressors for instance.

\paragraph{Acknowledgment}
We wish to acknowledge the PIA framework
(CGI, ANR) and the industrial members of the IRT
Saint Exupéry project R-Evol: Airbus, Liebherr,
Altran Technologies, Capgemini DEMS France,
CENAERO and CERFACS for their support,
financial funding and own knowledge. We are grateful to R\'eda Chha\"ibi (Institut de Math\'ematiques de Toulouse) for useful discussions on random matrices. We acknowledge Syver D\o ving Agdestein for a preliminary work on this topic, during its master internship. 

\section*{Declarations}

\paragraph{Replication of results}
All the details required for the replication of the
results presented in this paper are provided in sections \ref{sec4} and \ref{sec3}. A first implementation of the scalable problems (\ref{eq:martinspb}) and (\ref{TM_UMDO_initial}) is already available in the open source library GEMSEO: \url{https://gitlab.com/gemseo/dev/gemseo}. The implementation of the QP problems (\ref{QP_problem}), (\ref{QP_problem_margin}) and (\ref{QP_problem_proba}) should be released in GEMSEO on spring 2023.

\paragraph{Competing interests}
On behalf of all authors, the corresponding author states that there is no conflict of interest.

\bibliography{scalableUMDObenchmark}

\end{document}